\documentclass[preprint,11pt]{elsarticle}
\usepackage{lineno,hyperref}
\modulolinenumbers[5]

\usepackage{amsmath}
\usepackage{amsfonts,amsthm,amssymb}
\usepackage{amsfonts}
\usepackage{graphics}
\usepackage{pstricks}
\usepackage{graphicx}
\usepackage{cleveref}
\usepackage{extarrows,chngpage,array,float}     
\usepackage{amssymb}
\usepackage{latexsym,bm}
\newtheorem{theorem}{Theorem}[section]
\newtheorem{lemma}[theorem]{Lemma}

\newtheorem{corollary}[theorem]{Corollary}

\allowdisplaybreaks   

\begin{document}

\begin{frontmatter}

\title{Tutte polynomials of fan-like graphs with applications in benzenoid systems}

\author{Tianlong Ma, Xian'an Jin, Fuji Zhang\\
\small School of Mathematical Sciences\\[-0.8ex]
\small Xiamen University\\[-0.8ex]
\small P. R. China\\
\small\tt Email: tianlongma@aliyun.com, xajin@xmu.edu.cn\\
    \small ~~~~\tt fjzhang@xmu.edu.cn}

\begin{abstract}
We study the computation of the Tutte polynomials of fan-like graphs and obtain expressions of their Tutte polynomials via generating functions. As applications, Tutte polynomials, in particular, the number of spanning trees, of two kinds of benzenoid systems, i.e. pyrene chains and triphenylene chains, are obtained.
\end{abstract}


\end{frontmatter}


\section{Introduction}

Let $G=(V, E)$ be a graph with the set $V$ of vertices and the set $E$ of edges. The \emph{Tutte polynomial} of the graph $G$, denoted by $T(G;x, y)$, was introduced by Tutte \cite{Tutte} in 1954 as a generalization of chromatic
polynomials, which can be defined by the closed formula
\begin{eqnarray*}
T(G;x, y)=\sum_{A\subseteq E}(x-1)^{r(E)-r(A)}(y-1)^{|A|-r(A)},
\end{eqnarray*}
where $r(A)=|V|-\omega(V,A)$, and $\omega(V,A)$ denotes the number of connected components in
the graph $(V,A)$. It can also be calculated recursively, which can be viewed as an alternative definition, by using the following deletion-contraction formula, together with the initial condition that $T(G;x,y)=1$ if $E=\emptyset$. Let $e\in E$ be an edge of the graph $G$. Then
\begin{eqnarray*}
T(G; x, y)=\left \{
\begin{array}{ll}
xT(G/e;x, y),  &\text{ if } $e$ \text{ is a bridge};\\
yT(G-e;x, y),  &\text{ if } $e$ \text{ is a loop};\\
T(G-e;x, y)+T(G/e;x, y), &\text{ otherwise,}
\end{array}
\right.
\end{eqnarray*}
where $G-e$ denotes the graph obtained from $G$ by deleting the
edge $e$ and $G/e$ denotes the graph obtained from $G$ by contracting the
edge $e$, that is, deleting $e$ firstly and then identifying its two end-vertices into a new vertex.

The Tutte polynomial of a graph contains a large amount of information about the graph. For example, $T(G;1,1)$ counts the number of spanning trees of the graph $G$. It can be specialized to the chromatic and flow polynomials of a graph, the all terminal reliability probability of a network and the partition function of a $q$-state Potts model. Moreover, it can also be specialized to the Jones polynomial of an alternating knot or link and the weight enumerator of a linear code over $\mathrm{GF}(q)$.

However it is $\#P$ hard to compute the Tutte polynomial in general \cite{Jaeger}. Hence, various techniques (transfer matrix method, subgraph-decomposition trick, etc) have been developed to obtain the Tutte polynomial of many graph families appearing in different fields including the field of mathematical chemistry. For example, in \cite{Fath-Tabar}, a recursive formula of the Tutte polynomial of benzenoid chains was obtained. In \cite{Do}, the formula was extended to $k$-uniform polygonal chains and
when $k=6$, it is a benzenoid chain. For more general polycyclic chains of polygons, a general scheme was discussed in \cite{Dob} for computing many polynomials with the deletion-contraction property. In \cite{Gong}, an explicit expression for the Tutte polynomial of catacondensed benzenoid systems with exactly one branched hexagon was obtained. In \cite{Chen}, Tutte polynomials of alternating polycyclic chains were obtained. Recently, a reduction formula for Tutte polynomial
of any catacondensed benzenoid system was obtained by three classes of transfer matrices in \cite{Ren}.

In \cite{Li}, the authors studied sextet polynomials of hexagonal systems via generating functions, which motivates us to study the Tutte polynomial of benzenoid systems with recursive structure via generating functions. This is realized by computing the Tutte polynomial of fan-like graph families which are the planar duals of some benzenoid systems with repeated substructures. Fan-like graphs are also used to compute the Kauffman bracket polynomials of rational links, i.e. 2-bridge links in \cite{Jin1}. As applications, Tutte polynomials, in particular, the number of spanning trees, of two kinds of benzenoid systems, i.e. pyrene chains and triphenylene chains, are obtained.

\section{Preliminaries}

The first two theorems are well-known which can be found in some textbooks on graph theory such as \cite{Bollobas}.

\begin{theorem}\label{dual}
Let $D(G)$ be the planar dual of a plane graph $G$. Then
\begin{eqnarray}
T(G; x, y)=T(D(G); y,x).
\end{eqnarray}
\end{theorem}

\begin{theorem}\label{GH}
Let $G\ast H$ be a union of two graphs $G$ and $H$ which have only a common vertex $v$, as shown in Fig. 1. Then
\begin{eqnarray}
T(G\ast H; x, y)=T(G; x,y)T(H; x,y).
\end{eqnarray}
\end{theorem}
\begin{center}
\includegraphics{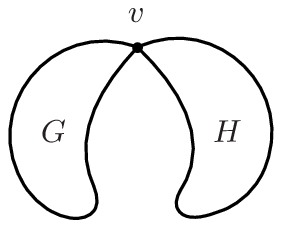}
\put(-75,-13){\mbox{\scriptsize\quad Fig. 1. $G\ast H$.}}
\end{center}

Let $S\subseteq V$, we use $G/S$ to denote the graph obtained from $G$ by identifying all the vertices in $S$ into a new vertex (all edges remain preserved). For convenience, in the following we sometimes abbreviate $T(G; x,y)$ into $T(G)$. The following lemma is contained as a special case in the general splitting-formula \cite{Neg} of the Tutte polynomial or 2-splitting formula for the Tutte polynomial of signed graphs \cite{Jin2}.
\begin{center}
\includegraphics{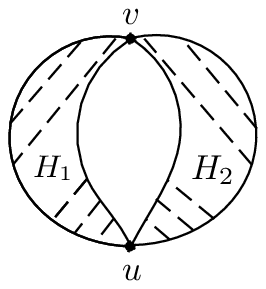}
\put(-67,-13){\mbox{\scriptsize\quad Fig. 2. $G$.}}
\end{center}
\begin{lemma}\label{H1H2}
Let $G$ be the union of two edge-disjoint connected graphs $H_1$ and $H_2$ having only two common vertices $v$ and $u$, as shown in Fig. 2. Let $H'_i=H_i/\{v,u\}$ for $i=1,2$. Then
\begin{align*}
T(G)=&\frac{1}{xy-x-y}[(y-1)T(H_1)T(H_2)+(x-1)T(H_1')T(H_2')\\
&-T(H_1)T(H_2')-T(H_1')T(H_2)].
\end{align*}
\end{lemma}

As an application of Theorem \ref{GH} and Lemma \ref{H1H2}, we have the following result.
\begin{corollary}\label{mT}
Let $G$ be the union of two edge-disjoint connected graphs $H_1$ and $H_2$ such that $V(H_1)\cap V(H_2)={v}$, $V(H_1)\cup V(H_2)=V(G)$ and $E(H_1)\cup E(H_2) =E(G-e)$, where $e=u_1u_2$, $u_1\in V(H_1)$ and $u_2\in V(H_2)$, as shown in Fig. 3.  Let $H'_i=H_i/\{v,u_i\}$ for $i=1,2$. Then
\begin{align*}
T(G)=&\frac{1}{xy-x-y}[(xy-x-1)T(H_1)T(H_2)+(x-1)T(H'_1)T(H'_2)\\
&-T(H_1)T(H'_2)-T(H'_1)T(H_2)].
\end{align*}
\end{corollary}

\begin{center}
\includegraphics{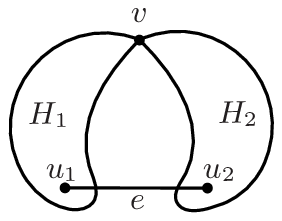}
\put(-67,-19){\mbox{\scriptsize\quad Fig. 3. $G$.}}
\end{center}
\begin{proof}
Applying the deletion-contraction formula to the edge $e$, we have $T(G)=T(G-e)+T(G/e)$.
By Theorem \ref{GH}, $T(G-e)=T(H_1)T(H_2)$. By Lemma \ref{H1H2}, we have 	
\begin{align*}
T(G/e)=&\frac{1}{xy-x-y}[(y-1)T(H_1)T(H_2)+(x-1)T(H'_1)T(H'_2)\\
&-T(H_1)T(H'_2)-T(H'_1)T(H_2)].
\end{align*}
Thus the Corollary holds.
\end{proof}

The following lemma is not difficult to prove and can be found in \cite{Li}.

\begin{lemma}\emph{\cite{Li}}\label{gflambda}
Let $f(x)$ be the generating function of the sequence $\{a_n\}^{\infty}_0$ and suppose that \[f(x)=\frac{1}{px^2-qx+1}.\] Then \[a_n=\frac{\lambda^{n+1}_1-\lambda^{n+1}_2}{\lambda_1-\lambda_2},\]
where \[\lambda_{1,2}=\frac{q\pm\sqrt{q^2-4p}}{2}\] are two roots of the equation
$\lambda^2-q\lambda+p=0$.
\end{lemma}

The $a_n$ also has the following combinatorial expression.

\begin{lemma}\label{gf}
Let $f(x)$ be the generating function of the sequence $\{a_n\}^{\infty}_0$ and suppose that $$f(x)=\frac{1}{px^2-qx+1}.$$ Then $$a_n=\sum_{j=0}^{\lfloor \frac{n}{2}\rfloor}
(-1)^j\binom{n-j}{j}q^{n-2j}p^{j}x^{n}.$$
\end{lemma}

\begin{proof}
\begin{align*}
\frac{1}{px^2-qx+1}&=\sum_{i=0}^{\infty}(q-px)^ix^i=\sum_{i=0}^{\infty}\sum_{j=0}^{i}(-1)^j
\binom{i}{j}q^{i-j}p^{j}x^{i+j}\\
&=\sum_{j=0}^{\infty}\sum_{i\geq j}(-1)^j
\binom{i}{j}q^{i-j}p^{j}x^{i+j}\\
&\xlongequal{i=n-j}
\sum_{j=0}^{\infty}\sum_{n\geq 2j}(-1)^j\binom{n-j}{j}q^{n-2j}p^{j}x^{n}\\
&=\sum_{n=0}^{\infty}\sum_{j=0}^{\lfloor \frac{n}{2}\rfloor}
(-1)^j\binom{n-j}{j}q^{n-2j}p^{j}x^{n}.\\
\end{align*}
Thus \[a_n=\sum_{j=0}^{\lfloor \frac{n}{2}\rfloor}
(-1)^j\binom{n-j}{j}q^{n-2j}p^{j}.\qedhere \]
\end{proof}

\section{Tutte polynomials of fan-like graphs}

We first recall the fan graph and the wheel graph. The join graph of two vertex disjoint graphs $G$
and $H$, denoted by $G\vee H$, is the graph with the vertex set $V(G)\cup V(H)$
and the edge set $E(G)\cup E(H)\cup \{uv\,|\, u\in V(G), v\in V(H)\}$. An $n$-\emph{fan}, denoted by $F_n$, is defined as $F_n=K_1\vee P_n$, where $K_1$ is the trivial graph, i.e. has one vertex and no edges, and $P_n$ is the path on $n$ vertices. An $n$-\emph{wheel}, denoted by $W_n$, is defined as $W_n=K_1\vee C_n$, where $C_n$ is the cycle on $n$ vertices.

Tutte polynomials of the fan and wheel graphs were easily obtained by establishing recursive relations using deletion-contraction formula. See \cite{Biggs} for the recursion of the wheel graphs. In this section we shall try to generalize them to the fan-like and wheel-like graphs.

\subsection{The first kind of fan-like graphs}

In this subsection, we shall define the first kind of fan-like graphs and study their Tutte polynomials and Tutte polynomials of related graphs including the wheel-like graphs.

Let $G$ be a connected graph with at least two distinct vertices $v$ and $u$, and $H_i$ be a copy of $G$ for $i=1,2,\ldots,n$. Then the fan-like graph $\mathcal{F}_{n}$ shown as Fig. 4 is defined to be the graph formed by identifying each vertex $v$ of $H_i$'s into a new vertex and connecting the vertex $u$ of $H_i$ to the vertex $u$ of $H_{i+1}$ by an edge $e_{i}$ for $i=1,2,\ldots,n-1$. $\mathcal{F}^{+}_n$ is defined to be the graph obtained from $\mathcal{F}_n$ by adding an edge connecting $v$ and $u$ of $H_1$, and $\mathcal{F}^{++}_n$ is defined to be the graph obtained from $\mathcal{F}^{+}_n$ by adding another edge connecting $v$ and $u$ of $H_n$, see Fig. 5. In particular, $\mathcal{F}_1^{+}=G^+$ and $\mathcal{F}_1^{++}=G^{++}$.
\begin{center}
\includegraphics{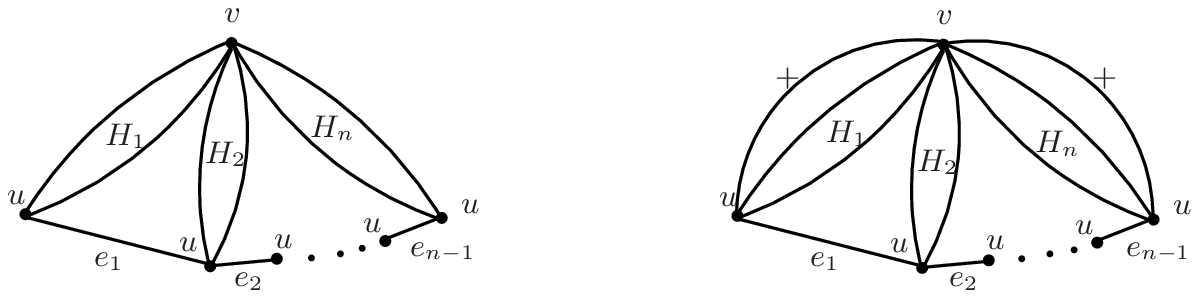}
	\put(-345,-10){\mbox{\scriptsize\quad Fig. 4. The fan-like graph $\mathcal{F}_{n}$.}}
	\put(-111,-10){\mbox{\scriptsize\quad Fig. 5. $\mathcal{F}^{++}_n$.}}
\end{center}

Now we present the first main result of this paper as follows. For convenience, we write a bivariate function $A(x,y)$ as $A$. In this subsection, we always assume that

\begin{align*}
	A=&\frac{x}{xy-x-y}((y-1)T(G)-T(G/\{v, u\})),\\
B=&\frac{1}{xy-x-y}((x-1)T(G/\{v, u\})-T(G)),\\
C=&\frac{1}{xy-x-y}((xy-y-1)T(G/\{v, u\})-T(G)),
	\end{align*}
	and $$\lambda_{1,2}=\frac{A+C\pm\sqrt{(A-C)^2+4AB}}{2}.$$

\begin{theorem}\label{GFn}
	Let $n\geq 1$ be an integer.
	Then
\begin{align}\label{F1}	
T(\mathcal{F}_{n})=T(G)\frac{\lambda^n_1-\lambda^n_2}{\lambda_1-\lambda_2}+((B-C)T(G)+BT(G/\{v,u\}))
	\frac{\lambda^{n-1}_1-\lambda^{n-1}_2}{\lambda_1-\lambda_2},
\end{align}
	or
	\begin{align}\label{F2}	
	T(\mathcal{F}_n)=&((B-C)T(G)+BT(G/\{v,u\}))\sum_{j=0}^{\lfloor \frac{n-2}{2}\rfloor}
	\binom{n-j-2}{j}(A+C)^{n-2j-2}\cdot \nonumber\\
	&(AB-AC)^{j}+T(G)\sum_{j=0}^{\lfloor \frac{n-1}{2}\rfloor}
	\binom{n-j-1}{j}(A+C)^{n-2j-1}(AB-AC)^{j}.
	\end{align}
	\end{theorem}

\begin{proof}
Let $R_1$ and $R_2$ be subgraphs of $\mathcal{F}_n$ induced by the sets $E(H_1)$ and $(\bigcup_{2\leq i\leq n} E(H_{i}))\cup \{e_1\}$, respectively. Clearly, $T(R_1)=T(G)$, $T(R_1/\{v, u\})=T(G/\{v, u\})$, $T(R_2)=xT(\mathcal{F}_{n-1})$ and $T(R_2/\{v, u\})=T(\mathcal{F}^{+}_{n-1})$, where $u\in V(H_1)$.
Let $S_1$ and $S_2$ be subgraphs of $\mathcal{F}^{+}_n$ induced by the sets $E(H_1)\cup\{+\}$ and $(\bigcup_{2\leq i\leq n} E(H_{i}))\cup \{e_1\}$, respectively. Clearly, $T(S_1)=T(G)+T(G/\{v, u\})$, $T(S_1/\{v, u\})=yT(G/\{v, u\})$, $T(S_2)=xT(\mathcal{F}_{n-1})$ and $T(S_2/\{v, u\})=T(\mathcal{F}^{+}_{n-1})$, where $u\in V(H_1)$.
By Lemma \ref{H1H2}, we have
	\begin{eqnarray}\label{F3}
	\left\{\begin {array}{l}
	T(\mathcal{F}_{n})=AT(\mathcal{F}_{n-1})+BT(\mathcal{F}_{n-1}^{+}), \\
	T(\mathcal{F}_{n}^{+})=AT(\mathcal{F}_{n-1})+CT(\mathcal{F}_{n-1}^{+}).
	\end{array}
	\right.
	\end{eqnarray}
	
	Let $$F(z)=\sum_{n\geq 1}T(\mathcal{F}_{n})z^n \text{ and } G(z)=\sum_{n\geq 1}T(\mathcal{F}_{n}^{+})z^n.$$
	By Equation (\ref{F3}), we have
	$$F(z)=\frac{((B-C)T(G)+BT(G/\{v,u\}))z^2+T(G)z}{A(C-B)z^2-(A+C)z+1}.$$
	Thus it is clear from Lemma \ref{gflambda} to obtain Equation (\ref{F1}). Moreover, by Lemma \ref{gf} we have
	\begin{align*}
	T(\mathcal{F}_n)=&((B-C)T(G)+BT(G/\{v,u\}))\sum_{j=0}^{\lfloor \frac{n-2}{2}\rfloor}
	\binom{n-j-2}{j}(A+C)^{n-2j-2}\cdot\\
	&(AB-AC)^{j}+T(G)\sum_{j=0}^{\lfloor \frac{n-1}{2}\rfloor}
	\binom{n-j-1}{j}(A+C)^{n-2j-1}(AB-AC)^{j}.\qedhere
	\end{align*}
\end{proof}

Note that $\mathcal{F}_{n}$ is reduced to the fan graph $F_n$ if $G=K_2$. As a corollary, we have

\begin{corollary}\label{K3}
Let $n\geq 1$ be an integer. Then
\begin{align*}	
T(F_{n})=x\frac{\lambda^n_1-\lambda^n_2}{\lambda_1-\lambda_2}+y(1-x)
	\frac{\lambda^{n-1}_1-\lambda^{n-1}_2}{\lambda_1-\lambda_2},
\end{align*}
or
	\begin{align*}
	T(F_n)=&y(1-x)\sum_{j=0}^{\lfloor \frac{n-2}{2}\rfloor}
	\binom{n-j-2}{j}(x+y+1)^{n-2j-2}(-xy)^{j}\\
	&+x\sum_{j=0}^{\lfloor \frac{n-1}{2}\rfloor}
	\binom{n-j-1}{j}(x+y+1)^{n-2j-1}(-xy)^{j},
	\end{align*}
where $$\lambda_{1,2}=\frac{x+y+1\pm\sqrt{4x+(y-x+1)^2}}{2}.$$
\end{corollary}

We now consider the Tutte polynomial of $\mathcal{F}^{++}_n$ by using similar method, which will be used in the Section 4.
\begin{theorem}\label{GFn+}
	Let $n\geq 1$ be an integer.
	Then
	\begin{align}\label{F+1}
    T(\mathcal{F}^{++}_n)=T(G^{++})\frac{\lambda^n_1-\lambda^n_2}{\lambda_1-\lambda_2}-yAT(G/\{v, u\})\frac{\lambda^{n-1}_1-\lambda^{n-1}_2}{\lambda_1-\lambda_2},
    \end{align}
	or
	\begin{align}\label{F+2}
		T(\mathcal{F}^{++}_n)=&T(G^{++})\sum_{j=0}^{\lfloor \frac{n-1}{2}\rfloor}
	\binom{n-j-1}{j}(A+C)^{n-2j-1}(AB-AC)^{j}\nonumber \\
	&-yAT(G/\{v, u\})\sum_{j=0}^{\lfloor \frac{n-2}{2}\rfloor}
	\binom{n-j-2}{j}(A+C)^{n-2j-2}A^{j}(B-C)^{j}.
	\end{align}
	\end{theorem}

\begin{proof}
Let $R_1$ and $R_2$ be subgraphs of $\mathcal{F}^{+}_n$ induced by the sets $(\bigcup_{1\leq i\leq n-1} E(H_{i}))\cup \{e_{n-1,+}\}$ and $E(H_n)$, respectively.
Clearly, $T(R_1)=xT(\mathcal{F}^{+}_{n-1})$, $T(R_1/\{v, u\})=T(\mathcal{F}^{++}_{n-1})$, $T(R_2)=T(G)$ and $T(R_2/\{v, u\})=T(G/\{v, u\})$, where $u\in V(H_n)$.
Let $S_1$ and $S_2$ be subgraphs of $\mathcal{F}^{++}_n$ induced by the sets $(\bigcup_{1\leq i\leq n-1} E(H_{i}))\cup \{e_{n-1}\}\cup\{+\}$ (the left one) and $E(H_n)\cup \{+\}$ (the right one), respectively.
Clearly, $T(S_1)=xT(\mathcal{F}^{+}_{n-1})$, $T(S_1/\{v, u\})=T(\mathcal{F}^{++}_{n-1})$, $T(S_2)= T(G)+T(G/\{v, u\})$ and $T(S_2/\{v, u\})=yT(G/\{v, u\})$, where $u\in V(H_n)$.
By Lemma \ref{H1H2}, we have
	\begin{eqnarray}\label{F+3}
	\left\{\begin {array}{l}
	T(\mathcal{F}^{+}_n)=AT(\mathcal{F}^{+}_{n-1})+BT(\mathcal{F}^{++}_{n-1}), \\
	T(\mathcal{F}^{++}_n)=AT(\mathcal{F}^{+}_{n-1})+CT(\mathcal{F}^{++}_{n-1}).
	\end{array}
	\right.
	\end{eqnarray}
	Let $$F(z)=\sum_{n\geq 1}T(\mathcal{F}^{+}_{n})z^n \text{ and } G(z)=\sum_{n\geq 1}T(\mathcal{F}_{n}^{++})z^n.$$
	By Equation (\ref{F+3}), we have
	$$G(z)=\frac{-yAT(G/\{v, u\})z^2+T(G^{++})z}{A(C-B)z^2-(A+C)z+1}.$$
	Thus it is clear from Lemma \ref{gflambda} to obtain the Equation (\ref{F+1}). Moreover, by Lemma \ref{gf} we have
	\begin{align*}
	T(\mathcal{F}^{++}_n)=&-yAT(G/\{v, u\})\sum_{j=0}^{\lfloor \frac{n-2}{2}\rfloor}
	\binom{n-j-2}{j}(A+C)^{n-2j-2}A^{j}(B-C)^{j}\\
	&+T(G^{++})\sum_{j=0}^{\lfloor \frac{n-1}{2}\rfloor}
	\binom{n-j-1}{j}(A+C)^{n-2j-1}(AB-AC)^{j}.\qedhere
	\end{align*}
\end{proof}

 We can further define the wheel-like graph as follows. The wheel-like graph $\mathcal{W}_n$ can be obtained from $\mathcal{F}_n$ by adding an edge $e_n$ connecting $u$ in $H_1$ and $u$ in $H_n$ and $n\geq 2$, see Fig. 6. If $H_i=K_2$ for $i=1,2,\ldots,n$ then $\mathcal{W}_n$ is reduced to the wheel graph $W_n$.

\begin{center}
\includegraphics{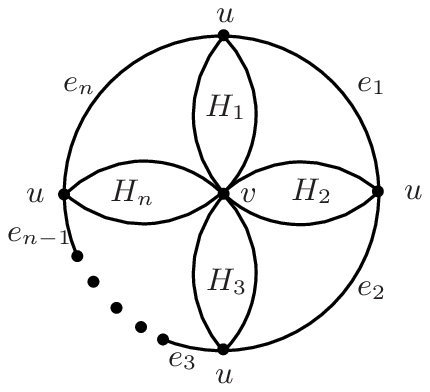}
	\put(-130,-16){\mbox{\scriptsize\quad Fig. 6. The wheel-like graph $\mathcal{W}_n$.}}
\end{center}
\begin{theorem}\label{GWn}
	Let $n\geq 3$ be an integer.
	Then
	\begin{align*} T(\mathcal{W}_n)=&\frac{A^{n-2}}{x^{n-2}}T(\mathcal{W}_2)+\sum_{i=2}^{n-1}\frac{A^{n-i-1}}{x^{n-i-1}}\Big[(A(B-C)T(G)+(1-y)ABT(G/\{v,u\}))\\
  &\sum_{j=0}^{\lfloor \frac{i-2}{2}\rfloor}
	\binom{i-j-2}{j}(A+C)^{i-2j-2}(AB-AC)^{j}+((A+B)T(G)\\
	&+(y+1)BT(G/\{v, u\}))\sum_{j=0}^{\lfloor \frac{i-1}{2}\rfloor}
	\binom{i-j-1}{j}(A+C)^{i-2j-1}(AB-AC)^{j}\Big].
	\end{align*}
	\end{theorem}

\begin{proof}
Let $S_1$ and $S_2$ be subgraphs of $\mathcal{W}_n$ induced by the sets $E(H_n)$ and $E(\mathcal{W}_n)\setminus E(H_n)$, respectively.
Clearly, $T(S_1)=T(G)$, $T(S_1/\{v, u\})=T(G/\{v, u\})$, $T(S_2)=xT(\mathcal{F}_{n-1})+T(\mathcal{W}_{n-1})$ and $T(S_2/\{v, u\})=T(\mathcal{F}^{++}_{n-1})$, where $u\in V(H_n)$. By Lemma \ref{H1H2}, we have
	\begin{align*}
  T(\mathcal{W}_n)=&\frac{A}{x}T(\mathcal{W}_{n-1})+AT(\mathcal{F}_{n-1})+BT(\mathcal{F}^{++}_{n-1})\\
   =&\frac{A^{n-2}}{x^{n-2}}T(\mathcal{W}_2)+\sum_{i=2}^{n-1}\frac{A^{n-i-1}}{x^{n-i-1}}(AT(\mathcal{F}_{i})+BT(\mathcal{F}^{++}_{i})).
   \end{align*}
   By Equations (\ref{F2}) and (\ref{F+2}) we have
  \begin{align*}
  T(\mathcal{W}_n)=&\frac{A^{n-2}}{x^{n-2}}T(\mathcal{W}_2)+\sum_{i=2}^{n-1}\frac{A^{n-i-1}}{x^{n-i-1}}(AT(\mathcal{F}_{i})+BT(\mathcal{F}^{++}_{i}))\\
  =&\frac{A^{n-2}}{x^{n-2}}T(\mathcal{W}_2)+\sum_{i=2}^{n-1}\frac{A^{n-i-1}}{x^{n-i-1}}\Big[(A(B-C)T(G)+ABT(G/\{v,u\}))\\
  &\sum_{j=0}^{\lfloor \frac{i-2}{2}\rfloor}
	\binom{i-j-2}{j}(A+C)^{i-2j-2}(AB-AC)^{j}\\
	&+AT(G)\sum_{j=0}^{\lfloor \frac{i-1}{2}\rfloor}
	\binom{i-j-1}{j}(A+C)^{i-2j-1}(AB-AC)^{j}\\
  &-yABT(G/\{v, u\})\sum_{j=0}^{\lfloor \frac{i-2}{2}\rfloor}
	\binom{i-j-2}{j}(A+C)^{i-2j-2}(AB-AC)^{j}\\
	&+BT(G^{++})\sum_{j=0}^{\lfloor \frac{i-1}{2}\rfloor}
	\binom{i-j-1}{j}(A+C)^{i-2j-1}(AB-AC)^{j}\Big]\\
  =&\frac{A^{n-2}}{x^{n-2}}T(\mathcal{W}_2)+\sum_{i=2}^{n-1}\frac{A^{n-i-1}}{x^{n-i-1}}\Big[(A(B-C)T(G)+(1-y)ABT(G/\{v,u\}))\\
  &\sum_{j=0}^{\lfloor \frac{i-2}{2}\rfloor}
	\binom{i-j-2}{j}(A+C)^{i-2j-2}(AB-AC)^{j}+(AT(G)+BT(G^{++}))\\
	&\sum_{j=0}^{\lfloor \frac{i-1}{2}\rfloor}
	\binom{i-j-1}{j}(A+C)^{i-2j-1}(AB-AC)^{j}\Big].
  \end{align*}
  Note that $T(G^{++})=T(G)+(y+1)T(G/\{v, u\})$. Thus the result holds.
\end{proof}

As an immediate consequence of Theorem \ref{GWn}, we have
\begin{corollary}
	Let $n\geq 3$ be an integer.
	Then
	\begin{align*}
T(W_n)=&\sum_{i=2}^{n-1}\Big[(xy(1-x-y))\sum_{j=0}^{\lfloor \frac{i-2}{2}\rfloor}
	\binom{i-j-2}{j}(x+y+1)^{i-2j-2}(-xy)^{j}\\
&+(x^2+x+y+y^2)\sum_{j=0}^{\lfloor \frac{i-1}{2}\rfloor}
	\binom{i-j-1}{j}(x+y+1)^{i-2j-1}(-xy)^{j}\Big]\\
&+x^2+x+xy+y+y^2.
	\end{align*}
\end{corollary}

\subsection{The second kind of fan-like graphs}

In the following, we consider another kind of more complicated fan-like graphs.

Let $G$ be a connected graph with at least three distinct vertices $v$, $u$ and $w$, and $H_i$ be a copy of $G$ for $i=1,2,\ldots,n$. Then the fan-like graph $\mathcal{G}_n$ is defined to be the graph formed by identifying each vertex $v$ of $H_i$'s into a new vertex and connecting the vertex $w$ of $H_i$ to the vertex $u$ of $H_{i+1}$ by an edge $e_i$ for $i=1,2,\ldots,n-1$, see Fig. 7. Moreover, we define $^{+}\mathcal{G}_n$ and $^{+}\mathcal{G}^{+}_n$ similarly. The graph $^{+}\mathcal{G}_n$ is obtained from $\mathcal{G}_n$ by adding an edge connecting $v$ and $u$ of $H_1$, and the graph $^{+}\mathcal{G}^{+}_n$ is obtained from $^{+}\mathcal {G}_n$ by adding an edge connecting $v$ and $w$ of $H_n$, see Fig. 8. In particular, $^{+}\mathcal{G}_1=^+G$ and $^{+}\mathcal{G}^{+}_1=^{+}G^{+}$.
\begin{center}
\includegraphics[width=12cm,height=3cm]{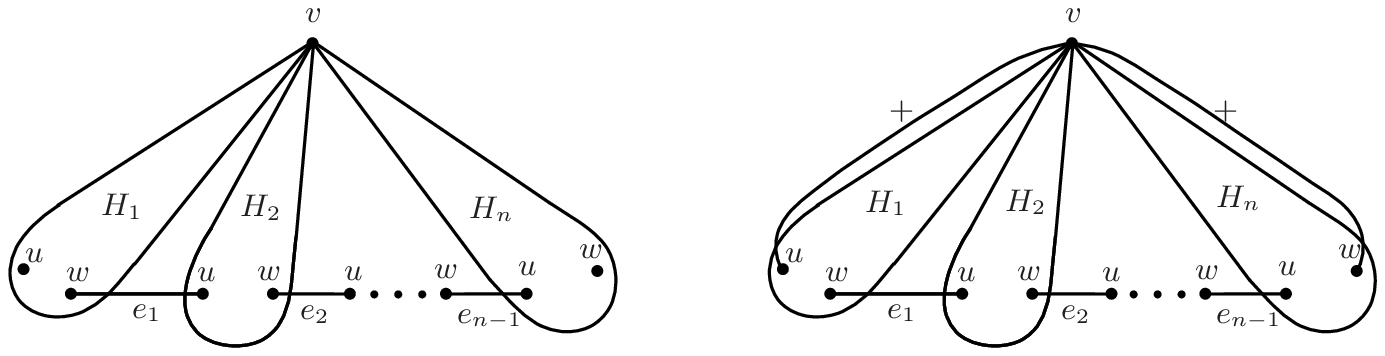}
\put(-330,-19){\mbox{\scriptsize\quad Fig. 7. The fan-like graph $\mathcal{G}_n$.}}
\put(-115,-19){\mbox{\scriptsize\quad Fig. 8. $^{+}\mathcal{G}^{+}_n$.}}
\end{center}

\begin{theorem}\label{Gn}
Let $n\geq 1$ be an integer.
Then
\begin{align}\label{G1}
T(\mathcal{G}_n)=T(G)\frac{\lambda^n_1-\lambda^n_2}{\lambda_1-\lambda_2}+(BT(G/\{v,u\})
-DT(G))\frac{\lambda^{n-1}_1-\lambda^{n-1}_2}{\lambda_1-\lambda_2},
\end{align}
or
\begin{align*}
T(\mathcal{G}_n)=&(BT(G/\{v,u\})-DT(G))\sum_{j=0}^{\lfloor \frac{n-2}{2}\rfloor}
\binom{n-j-2}{j}(A+D)^{n-2j-2}(BC-AD)^{j} \\
&+T(G)\sum_{j=0}^{\lfloor \frac{n-1}{2}\rfloor}
\binom{n-j-1}{j}(A+D)^{n-2j-1}(BC-AD)^{j},
\end{align*}
where\begin{align*}
A=&\frac{1}{xy-x-y}((xy-x-1)T(G)-T(G/\{v, w\})),\\
B=&\frac{1}{xy-x-y}((x-1)T(G/\{v, w\})-T(G)),\\
C=&\frac{1}{xy-x-y}((xy-x-1)T(G/\{v, u\})-T(G/\{v, u, w\})),\\
D=&\frac{1}{xy-x-y}((x-1)T(G/\{v, u, w\})-T(G/\{v, u\})),
\end{align*}
and $$\lambda_{1,2}=\frac{A+D\pm\sqrt{(A-D)^2+4BC}}{2}.$$
\end{theorem}

\begin{proof}
Let $\mathcal{G}'_n=\mathcal{G}_n/\{v, u\}$, where $u\in V(H_1)$.
From Corollary \ref{mT}, conducting deletion-contraction operation for the edge $e_1=wu$ connecting
$H_1$ and $H_2$ in $\mathcal{G}_n$ and $\mathcal{G}'_n$, respectively, we have
\begin{eqnarray}\label{G2}
	\left\{\begin {array}{l}
	T(\mathcal{G}_n)=AT(\mathcal{G}_{n-1})+BT(\mathcal{G}'_{n-1}), \\
	T(\mathcal{G}'_{n})=CT(\mathcal{G}_{n-1})+DT(\mathcal{G}'_{n-1}).
\end{array}
\right.
\end{eqnarray}

Let $$F(z)=\sum_{n\geq 1}T(\mathcal{G}_n)z^n \text{ and } G(z)=\sum_{n\geq 1}T(\mathcal{G}'_n)z^n.$$

By Equation (\ref{G2}), we have
$$F(z)=\frac{(BT(G/\{v,u\})-DT(G))z^2+T(G)z}{(AD-BC)z^2-(A+D)z+1}.$$
Thus it is clear from Lemma \ref{gflambda} to obtain Equation (\ref{G1}).
Moreover, by Lemma \ref{gf} we have
\begin{align*}
T(\mathcal{G}_n)=&(BT(G/\{v,u\})-DT(G))\sum_{j=0}^{\lfloor \frac{n-2}{2}\rfloor}
\binom{n-j-2}{j}(A+D)^{n-2j-2}(BC-AD)^{j}\\
&+T(G)\sum_{j=0}^{\lfloor \frac{n-1}{2}\rfloor}
\binom{n-j-1}{j}(A+D)^{n-2j-1}(BC-AD)^{j}.\qedhere
\end{align*}
\end{proof}

The next goal is to get the Tutte polynomials of $^{+}\mathcal{G}^{+}_n$.

\begin{theorem}\label{GGn+}
Let $n\geq 1$ be an integer. Then
\begin{align}\label{G+1}
T(^{+}\mathcal{G}^{+}_n)=T(^{+}G^{+})\frac{\lambda^n_1-\lambda^n_2}{\lambda_1-\lambda_2}+(CT(^{+}G)-AT(^{+}G^{+}))\frac{\lambda^{n-1}_1-\lambda^{n-1}_2}{\lambda_1-\lambda_2},
\end{align}
	or
	\begin{align*}
	T(^{+}\mathcal{G}^{+}_n)=&(CT(^{+}G)-AT(^{+}G^{+}))\sum_{j=0}^{\lfloor \frac{n-2}{2}\rfloor}
	\binom{n-j-2}{j}(A+D)^{n-2j-2}(BC-AD)^{j}\\
	&+T(^{+}G^{+})\sum_{j=0}^{\lfloor \frac{n-1}{2}\rfloor}
	\binom{n-j-1}{j}(A+D)^{n-2j-1}(BC-AD)^{j},
	\end{align*}
	where
	\begin{align*}
	A=&\frac{x}{xy-x-y}((y-1)T(G)-T(G/\{v, u\})),\\
	B=&\frac{1}{xy-x-y}((x-1)T(G/\{v, u\})-T(G)),\\
	C=&\frac{x}{xy-x-y}((y-1)T(G^{+})-T(G^{+}/\{v, u\})),\\
	D=&\frac{1}{xy-x-y}((x-1)T(G^{+}/\{v, u\})-T(G^{+})),
	\end{align*}
	and $$\lambda_{1,2}=\frac{A+D\pm\sqrt{(A-D)^2+4BC}}{2}.$$
\end{theorem}
\begin{proof}
	Let $R_1$ and $R_2$ be subgraphs of $^{+}\mathcal{G}_n$ induced by the sets $(\bigcup_{1\leq i\leq n-1} E(H_{i}))\cup \{e_{n-1},+\}$ and $E(H_n)$, respectively.
Clearly, $T(R_1)=xT(^{+}\mathcal{G}_{n-1})$, $T(R_1/\{v, u\})=T(^{+}\mathcal{G}^{+}_{n-1})$, $T(R_2)=T(G)$ and $T(R_2/\{v, u\})=T(G/\{v, u\})$, where $u\in V(H_n)$.
Let $S_1$ and $S_2$ be subgraphs of $^{+}\mathcal{G}^{+}_n$ induced by the sets $(\bigcup_{1\leq i\leq n-1} E(H_{i}))\cup \{e_{n-1}\}\cup\{+\}$ (the left one) and $E(H_n)\cup \{+\}$ (the right one), respectively.
Clearly, $T(S_1)=xT(^{+}\mathcal{G}_{n-1})$, $T(S_1/\{v, u\})=T(^{+}\mathcal{G}^{+}_{n-1})$, $T(S_2)= T(G^{+})$ and $T(S_2/\{v, u\})=T(G^{+}/\{v, u\})$, where $u\in V(H_n)$.
From Lemma \ref{H1H2}, we have

\begin{eqnarray}\label{G+2}
	\left\{\begin {array}{l}
	T(\mathcal{G}^{+}_n)=AT(\mathcal{G}^{+}_{n-1})+BT(\mathcal{G}^{++}_{n-1}), \\
	T(\mathcal{G}^{++}_{n})=CT(\mathcal{G}^{+}_{n-1})+DT(\mathcal{G}^{++}_{n-1}).
\end{array}
\right.
\end{eqnarray}

Let $$F(z)=\sum_{n\geq 1}T(^{+}\mathcal{G}_n)z^n \text{ and } G(z)=\sum_{n\geq 1}T(^{+}\mathcal{G}^{+}_n)z^n.$$
By Equation (\ref{G+2}), we have
$$G(z)=\frac{(CT(G^{+})-AT(G^{++}))z^2+T(G^{++})z}{(AD-BC)z^2-(A+D)z+1}.$$
Thus it is clear from Lemma \ref{gflambda} to obtain Equation (\ref{G+1}).
Moreover, by Lemma \ref{gf} we have	
	\begin{align*}
	T(^{+}\mathcal{G}^{+}_n)=&(CT(^{+}G)-AT(^{+}G^{+}))\sum_{j=0}^{\lfloor \frac{n-2}{2}\rfloor}
	\binom{n-j-2}{j}(A+D)^{n-2j-2}(BC-AD)^{j}\\
	&+T(^{+}G^{+})\sum_{j=0}^{\lfloor \frac{n-1}{2}\rfloor}
	\binom{n-j-1}{j}(A+D)^{n-2j-1}(BC-AD)^{j}.\qedhere
	\end{align*}
\end{proof}

This theorem will be used to obtain the Tutte polynomials of two families of benzenoid systems in the next section.

\section{Applications}

\emph{A benzenoid system} is a finite connected plane graph without cut vertices in which every
interior face is bounded by a regular hexagon of side length $1$. A benzenoid system is called \emph{catacondensed} if there are no three hexagons sharing one common vertex.
As mentioned in the Introduction, the Tutte polynomials of several classes of benzenoid systems are derived. One idea computing the Tutte polynomials of benzenoid systems is to compute the Tutte polynomials of their duals.

\subsection{Benzenoid chains}

 A \emph{benzenoid chain}, denoted by $L_n$, is a catacondensed benzenoid system in which no hexagon is adjacent to three hexagons and there exist exact two hexagons that are adjacent to one hexagon, see Fig. 9.
\begin{center}
\includegraphics{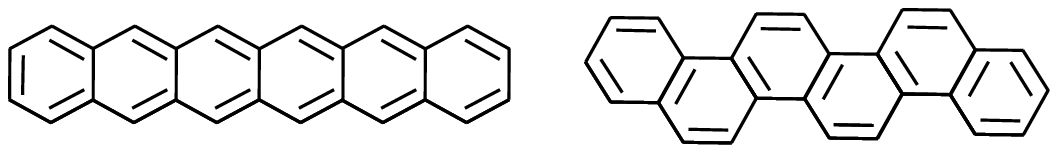}
\put(-250,-19){\mbox{\scriptsize\quad Fig. 9. Two benzenoid chains with $6$ hexagons.}}
\end{center}
In\cite{Fath-Tabar} and \cite{Do}, the Tutte polynomial of benzenoid chains was obtained.
\begin{theorem}\emph{\cite{Fath-Tabar,Do}}
Let $n\geq 1$ be an integer. Then
\begin{align*}
    T(L_n)=A\left (\frac{J+\sqrt{\Delta}}{2}\right )^n+B\left (\frac{J-\sqrt{\Delta}}{2}\right )^n,
    \end{align*}
 where
 \begin{align*}
J=&\sum^{4}_{i=0}x^i+y,
\Delta=\left (\sum^{4}_{i=0}x^i\right )^2+2y\sum^{4}_{i=0}x^i+y^2-4x^4y,\\
A=&\frac{2T(L_2)-T(L_1)(J-\sqrt{\Delta})}{\Delta+J\sqrt{\Delta}},
B=\frac{2T(L_2)-T(L_1)(J+\sqrt{\Delta})}{\Delta-J\sqrt{\Delta}}.
\end{align*}
\end{theorem}

As an application of Theorem \ref{GFn+}, we can also obtain $T(L_n)$ in another two forms.

\begin{corollary}\label{Ln}
Let $n\geq 1$ be an integer. Then
\begin{align*}
    T(L_n)=\left(\sum^{5}_{i=1}x^i+y\right)\frac{\lambda^n_1-\lambda^n_2}{\lambda_1-\lambda_2}-x^5y\frac{\lambda^{n-1}_1-\lambda^{n-1}_2}{\lambda_1-\lambda_2},
    \end{align*}
    where
\[
 \lambda_{1,2}=\frac{\sum^{4}_{i=0}x^i+y\pm\sqrt{4y\sum^{3}_{i=0}x^i+\left(\sum^{4}_{i=0}x^i-y\right)^2}}{2}
\]
	or
	\begin{align*}
		T(L_n)=&\left(\sum^{5}_{i=1}x^i+y\right)\sum_{j=0}^{\lfloor \frac{n-1}{2}\rfloor}
	(-x^4y)^{j}\binom{n-j-1}{j}\left(\sum^{4}_{i=0}x^i+y\right)^{n-2j-1} \\
	&-x^5y\sum_{j=0}^{\lfloor \frac{n-2}{2}\rfloor}
	(-x^4y)^{j}\binom{n-j-2}{j}\left(\sum^{4}_{i=0}x^i+y\right)^{n-2j-2}.
	\end{align*}
	
\end{corollary}

\subsection{Pyrene chains}

Let $R_n$ denote the pyrene chain as shown in Fig. 10. In \cite{Ohkami}, the recursive relation of the sextet polynomial for several classes of benzenoid systems including pyrene chains was obtained. Recently, zeros of sextet polynomials for pyrene chains were analyzed in \cite{Li} and the forcing and anti-forcing polynomials of perfect matchings of pyrene chains were studied in \cite{Deng}.
\begin{center}
\includegraphics{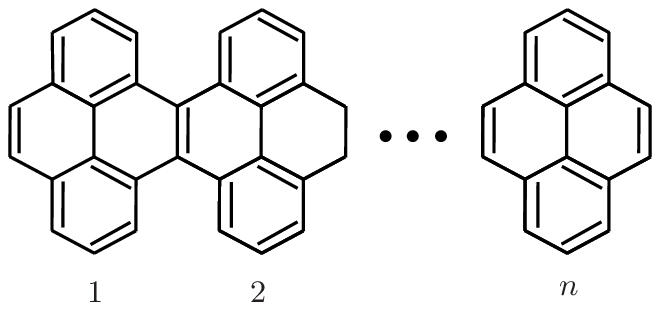}
\put(-170,-19){\mbox{\scriptsize\quad Fig. 10. The pyrene chain $R_n$.}}
\end{center}

Note $D(R_n)=^{+}\mathcal{G}^{+}_n$ when $G$ is the graph as shown in Fig. 11 (the label near the edge denotes the number of parallel edges).
\begin{center}
\includegraphics{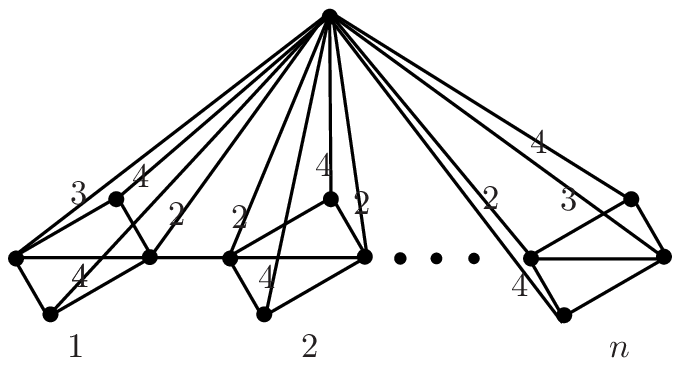}
\hspace{2cm}
\includegraphics{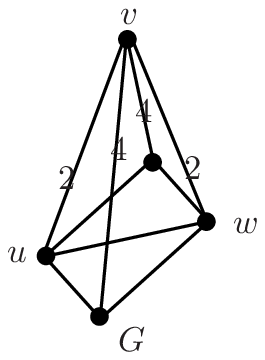}
\put(-280,-19){\mbox{\scriptsize\quad Fig. 11. The dual of $R_n$ and the corresponding $G$.}}
\end{center}

As an application of Theorem \ref{GGn+}, we have

\begin{corollary}\label{Rn}
Let $n\geq 1$ be an integer. Then $$T(R_n)=(I+2J+K)\frac{\lambda^n_1-\lambda^n_2}{\lambda_1-\lambda_2}+(C(I+J)-A(I+2J+K))\frac{\lambda^{n-1}_1-\lambda^{n-1}_2}{\lambda_1-\lambda_2},$$
	or
	\begin{align*}
	&T(R_n)=(I+2J+K)\sum_{j=0}^{\lfloor \frac{n-1}{2}\rfloor} \binom{n-j-1}{j}(A+D)^{n-2j-1}(BC-AD)^{j}\\
&+(C(I+J)-A(I+2J+K))\sum_{j=0}^{\lfloor \frac{n-2}{2}\rfloor}
	\binom{n-j-2}{j}(A+D)^{n-2j-2}(BC-AD)^{j},
	\end{align*}
	where
	\begin{align*}
	A=&\frac{y((x-1)I-J)}{xy-x-y},\quad C=\frac{y((x-1)(I+J)-(J+K))}{xy-x-y},\\
    B=&\frac{(y-1)J-I}{xy-x-y},\qquad D=\frac{(y-1)(J+K)-(I+J)}{xy-x-y},
	\end{align*}
$I=x^{13}+4x^{12}+10x^{11}+20x^{10}+2x^9y+33x^9+8x^8y+46x^8+18x^7y+56x^7+x^6y^2+31x^6y+60x^6+6x^5y^2+42x^5y+56x^5+11x^4y^2+49x^4y+44x^4+2x^3y^3+17x^3y^2+44x^3y+29x^3+2x^2y^3+17x^2y^2+34x^2y+15x^2+4xy^3+17xy^2+19xy+4x+y^4+5y^3+8y^2+4y$,\\
$J=x^2(x^{12}+3x^{11}+6x^{10}+10x^9+x^8y+14x^8+4x^7y+16x^7+7x^6y+16x^6+10x^5y+14x^5+2x^4y^2+11x^4y+10x^4+2x^3y^2+10x^3y+6x^3+3x^2y^2+7x^2y+3x^2+3xy^2+4xy+x+y^3+2y^2+y)$,\\
$K=x^5(x^5+x^4+x^3+x^2+x+y)^2$,\\
	and $$\lambda_{1,2}=\frac{A+D\pm\sqrt{(A-D)^2+4BC}}{2}.$$
\end{corollary}

\subsection{Triphenylene chains}

Let $T_n$ denote the triphenylene chain as shown in Fig. 12. Note that $D(T_n)=^{+}\mathcal{G}^{+}_n$ when $G$ is the graph as shown in Fig. 13.
\begin{center}
\includegraphics{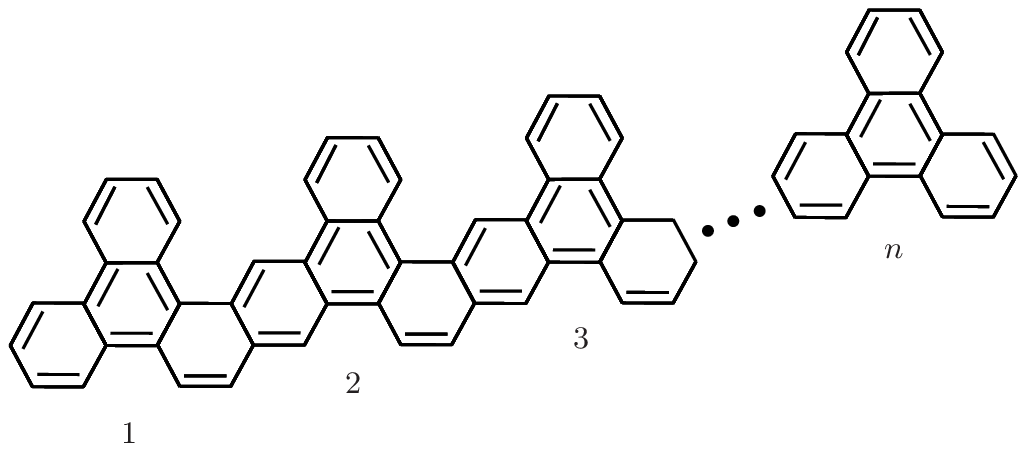}
\put(-220,-19){\mbox{\scriptsize\quad Fig. 12. The triphenylene chain $T_n$.}}
\end{center}
\begin{center}
\includegraphics{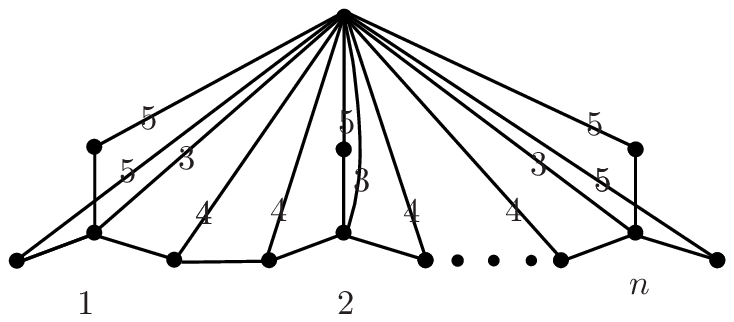}
\hspace{2cm}
\includegraphics{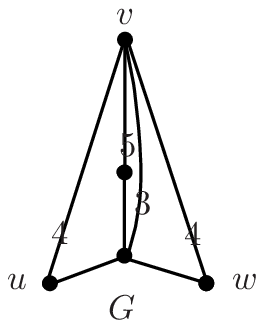}
\put(-280,-19){\mbox{\scriptsize\quad Fig. 13. The dual of $T_n$ and the corresponding $G$}}
\end{center}

As another application of Theorem \ref{GGn+}, we have

\begin{corollary}\label{Tn}
Let $n\geq 1$ be an integer. Then $$T(T_n)=(I+2J+K)\frac{\lambda^n_1-\lambda^n_2}{\lambda_1-\lambda_2}+(C(I+J)-A(I+2J+K))\frac{\lambda^{n-1}_1-\lambda^{n-1}_2}{\lambda_1-\lambda_2},$$
	or
	\begin{align*}
	&T(T_n)=(I+2J+K)\sum_{j=0}^{\lfloor \frac{n-1}{2}\rfloor}	\binom{n-j-1}{j}(A+D)^{n-2j-1}(BC-AD)^{j}\\
&+(C(I+J)-A(I+2J+K))\sum_{j=0}^{\lfloor \frac{n-2}{2}\rfloor}
	\binom{n-j-2}{j}(A+D)^{n-2j-2}(BC-AD)^{j},
	\end{align*}
	where
	\begin{align*}
	A=&\frac{y((x-1)I-J)}{xy-x-y},\quad C=\frac{y((x-1)(I+J)-(J+K))}{xy-x-y},\\
    B=&\frac{(y-1)J-I}{xy-x-y},\qquad D=\frac{(y-1)(J+K)-(I+J)}{xy-x-y},
	\end{align*}
$I=y^4+y^3x^4+3y^3x^3+4y^3x^2+4y^3x+3y^3+3y^2x^7+9y^2x^6+15y^2x^5+20y^2x^4+ 21y^2x^3+15y^2x^2+9y^2x+3y^2+2yx^{11}+8yx^{10}+18yx^9+32yx^8+46yx^7+53yx^6+52yx^5+ 43yx^4+28yx^3+15yx^2+6yx+y+x^{15}+4x^{14}+10x^{13}+20x^{12}+33x^{11}+46x^{10}+56x^9+60x^8+ 56x^7+46x^6+33x^5+20x^4+10x^3+4x^2+x$,\\
$J=x^4(y^3+y^2x^4+3y^2x^3+3y^2x^2+3y^2x+2y^2+yx^8+4yx^7+7yx^6+10yx^5+12yx^4+10yx^3+ 7yx^2+4yx+y+x^{12}+3x^{11}+6x^{10}+10x^9+14x^8+16x^7+16x^6+14x^5+10x^4+6x^3+3x^2+x)$,\\
$K=x^8(y+x+(x^4+x^3+x^2+x+y)^2+x^2+x^3+x^4+x^5+x^6+x^7+x^8+x^9)$,\\
	and $$\lambda_{1,2}=\frac{A+D\pm\sqrt{(A-D)^2+4BC}}{2}.$$
\end{corollary}

\subsection{Number of spanning trees}

We denoted by $\tau(G)=T(G; 1, 1)$ the number of spanning trees of a graph $G$. By Corollary \ref{Ln}, we have
\begin{corollary}\emph{\cite{Farrell, Gutman}}
	Let $n\geq 1$ be an integer. Then
$$\tau(L_{n})=\frac{4+3\sqrt{2}}{8}(3+2\sqrt{2})^{n}+\frac{4 -3\sqrt{2}}{8}(3-2\sqrt{2})^{n}.$$
\end{corollary}

By Corollaries \ref{Rn} and \ref{Tn}, we have
\begin{corollary}\label{ttRn}
Let $n\geq 1$ be an integer. Then
\begin{eqnarray*}
& &\tau(R_n)=\frac{240+47\sqrt{30}}{480}(528+96\sqrt{30})^{n}+\frac{240-47\sqrt{30}}{480}(528-96\sqrt{30})^{n},\\
& &\tau(T_n)=\frac{1329265+1223\sqrt{1329265}}{2658530}\left(\frac{1153+\sqrt{1329265}}{2}\right )^{n}+\\
& &\ \ \ \ \ \ \ \ \ \ \ \frac{1329265-1223\sqrt{1329265}}{2658530}\left(\frac{1153-\sqrt{1329265}}{2}\right )^{n}.
\end{eqnarray*}
\end{corollary}

\section{Discussions}

In this paper we mainly obtain expressions of Tutte polynomial of two families of fan-like graphs. As applications, the Tutte polynomials, in particular the number of spanning trees, of several families of recursive benzenoid systems such as pyrene chains and triphenylene chains are obtained, which as far as we know are both unknown. Some numerical results are listed in the Appendixes 1-3. It is interesting to find that $\tau(T_n)$ is greater than $\tau(R_n)$, and it seems that the difference becomes bigger and bigger when $n$ increases. We also note that dual graphs of some corona-condensed hexagonal systems are bipyramid-like graphs. For example, the Kekulene as shown in Fig. 14 and the primitive coronoid in \cite{Zhang}.  The Tutte polynomials of cones over a graph $G$, i.e. $K_1\vee G$ and bipyramids $\overline{K_2}\vee C_n$ was studied long long ago, see \cite{Biggsb}. In order to compute the Tutte polynomials of corona-condensed hexagonal systems, the computation of the Tutte polynomials of bipyramid-like graphs is deserved for further study.

\begin{center}
\includegraphics{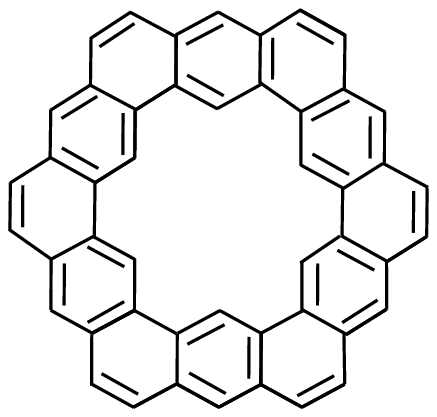}
\put(-235,-19){\mbox{\scriptsize\quad Fig. 14. Kekulene.}}
\end{center}

\section*{References}


\newpage
\textbf{Appendix 1.} Tutte polynomials of pyrene chains.

\begin{adjustwidth}{-1cm}{-1cm}
\centering
\begin{tabular}{|p{.1\textwidth}| p{.8\textwidth}| m{.7\textwidth}|}
\hline

&\qquad\qquad\qquad\qquad\qquad\qquad\quad $T(R_n,x,y)$ \\
\hline

$n=1$ &
\scriptsize{$x^{15} + 4x^{14} + 10x^{13} + 20x^{12} + 35x^{11} + 4x^{10}y + 52x^{10} + 12x^9y + 68x^9 + 24x^8y + 80x^8 + 40x^7y + 85x^7 + 5x^6y^2 + 55x^6y + 80x^6 + 11x^5y^2 + 62x^5y + 68x^5 + 17x^4y^2 + 63x^4y + 50x^4 + 2x^3y^3 + 23x^3y^2 + 52x^3y + 31x^3 + 4x^2y^3 + 21x^2y^2 + 36x^2y + 15x^2 + 4xy^3 + 17xy^2 + 19xy + 4x + y^4 + 5y^3 + 8y^2 + 4y$ }\\
\hline

$n=2$ &
\scriptsize{$x^{29} + 8x^{28} + 36x^{27} + 120x^{26} + 330x^{25} + 8x^{24}y + 784x^{24} + 56x^{23}y + 1652x^{23}+ 224x^{22}y + 3144x^{22} + 672x^{21}y + 5475x^{21} + 11x^{20}y^2 + 1669x^{20}y + 8800x^{20} + 83x^{19}y^2 + 3574x^{19}y + 13140x^{19} + 333x^{18}y^2 + 6773x^{18}y + 18316x^{18} + 4x^{17}y^3 + 973x^{17}y^2 + 11560x^{17}y + 23916x^{17} + 30x^{16}y^3 + 2316x^{16}y^2 + 17976x^{16}y + 29318x^{16} + 144x^{15}y^3 + 4698x^{15}y^2 + 25660x^{15}y + 33788x^{15} + 2x^{14}y^4 + 470x^{14}y^3 + 8336x^{14}y^2 + 33812x^{14}y + 36624x^{14} + 12x^{13}y^4 + 1180x^{13}y^3 + 13188x^{13}y^2 + 41266x^{13}y + 37323x^{13} + 59x^{12}y^4 + 2455x^{12}y^3 + 18807x^{12}y^2 + 46729x^{12}y + 35714x^{12} + 202x^{11}y^4 + 4360x^{11}y^3 + 24348x^{11}y^2 + 49112x^{11}y + 32016x^{11} + 2x^{10}y^5 + 512x^{10}y^4 + 6768x^{10}y^3 + 28732x^{10}y^2 + 47842x^{10}y + 26796x^{10} + 24x^9y^5 + 1049x^9y^4 + 9282x^9y^3 + 30924x^9y^2 + 43076x^9y + 20838x^9 + 80x^8y^5 + 1769x^8y^4 + 11319x^8y^3 + 30333x^8y^2 + 35685x^8y + 14950x^8 + 188x^7y^5 + 2556x^7y^4 + 12332x^7y^3 + 27002x^7y^2 + 26990x^7y + 9796x^7 + 10x^6y^6 + 362x^6y^5 + 3190x^6y^4 + 11938x^6y^3 + 21626x^6y^2 + 18434x^6y + 5780x^6 + 28x^5y^6 + 532x^5y^5 + 3409x^5y^4 + 10194x^5y^3 + 15404x^5y^2 + 11188x^5y + 3005x^5 + 53x^4y^6 + 655x^4y^5 + 3151x^4y^4 + 7583x^4y^3 + 9556x^4y^2 + 5876x^4y + 1330x^4 + 4x^3y^7 + 89x^3y^6 + 670x^3y^5 + 2438x^3y^4 + 4758x^3y^3 + 4990x^3y^2 + 2562x^3y + 473x^3 + 6x^2y^7 + 91x^2y^6 + 525x^2y^5 + 1526x^2y^4 + 2420x^2y^3 + 2071x^2y^2 + 857x^2y + 120x^2 + 8xy^7 + 81xy^6 + 336xy^5 + 733xy^4 + 894xy^3 + 592xy^2 + 184xy + 16x + y^8 + 11y^7 + 51y^6 + 129y^5 + 192y^4 + 168y^3 + 80y^2 + 16y$} \\
\hline
\end{tabular}
\end{adjustwidth}

\vspace{0.5cm}
\textbf{Appendix 2.} Tutte polynomials of triphenylene chains.
\begin{adjustwidth}{-1cm}{-1cm}
\centering
\begin{tabular}{|p{.1\textwidth}| p{.8\textwidth}| m{.7\textwidth}|}
\hline

&\qquad\qquad\qquad\qquad\qquad\qquad\quad $T(T_n,x,y)$ \\
\hline
$n=1$ &
\scriptsize {$x^{17} + 4x^{16} + 10x^{15} + 20x^{14} + 35x^{13} + 4x^{12}y + 52x^{12} + 12x^{11}y + 68x^{11} + 24x^{10}y + 80x^{10} + 40x^9y + 85x^9 + 3x^8y^2 + 57x^8y + 80x^8 + 9x^7y^2 + 66x^7y + 68x^7 + 15x^6y^2 + 67x^6y + 52x^6 + 21x^5y^2 + 60x^5y + 35x^5 + 3x^4y^3 + 24x^4y^2 + 45x^4y + 20x^4 + 3x^3y^3 + 21x^3y^2 + 28x^3y + 10x^3 + 4x^2y^3 + 15x^2y^2 + 15x^2y + 4x^2 + 4xy^3 + 9xy^2 + 6xy + x + y^4 + 3y^3 + 3y^2 + y$ }\\
\hline

$n=2$ &
\scriptsize {$x^{33} + 8x^{32} + 36x^{31} + 120x^{30} + 330x^{29} + 8x^{28}y + 784x^{28} + 56x^{27}y + 1652x^{27}+ 224x^{26}y + 3144x^{26} + 672x^{25}y + 5475x^{25} + 7x^{24}y^2 + 1673x^{24}y + 8800x^{24} + 63x^{23}y^2 + 3598x^{23}y + 13140x^{23} + 273x^{22}y^2 + 6853x^{22}y + 18320x^{22} + 833x^{21}y^2 + 11768x^{21}y + 23940x^{21} + 8x^{20}y^3 + 2050x^{20}y^2 + 18438x^{20}y + 29400x^{20} + 66x^{19}y^3 + 4294x^{19}y^2 + 26558x^{19}y + 34000x^{19} + 272x^{18}y^3 + 7858x^{18}y^2 + 35378x^{18}y + 37080x^{18} + 780x^{17}y^3 + 12818x^{17}y^2 + 43778x^{17}y + 38165x^{17} + 7x^{16}y^4 + 1793x^{16}y^3 + 18893x^{16}y^2 + 50463x^{16}y + 37080x^{16} + 60x^{15}y^4 + 3478x^{15}y^3 + 25358x^{15}y^2 + 54278x^{15}y + 34000x^{15} + 218x^{14}y^4 + 5842x^{14}y^3 + 31166x^{14}y^2 + 54538x^{14}y + 29400x^{14} + 551x^{13}y^4 + 8668x^{13}y^3 + 35228x^{13}y^2 + 51198x^{13}y + 23940x^{13} + 6x^{12}y^5 + 1123x^{12}y^4 + 11515x^{12}y^3 + 36693x^{12}y^2 + 44863x^{12}y + 18320x^{12} + 46x^{11}y^5 + 1918x^{11}y^4 + 13770x^{11}y^3 + 35228x^{11}y^2 + 36638x^{11}y + 13140x^{11} + 132x^{10}y^5 + 2800x^{10}y^4 + 14888x^{10}y^3 + 31166x^{10}y^2 + 27818x^{10}y + 8800x^{10} + 272x^9y^5 + 3577x^9y^4 + 14596x^9y^3 + 25358x^9y^2 + 19558x^9y + 5475x^9 + 6x^8y^6 + 462x^8y^5 + 4039x^8y^4 + 12957x^8y^3 + 18893x^8y^2 + 12663x^8y + 3144x^8 + 25x^7y^6 + 646x^7y^5 + 4019x^7y^4 + 10370x^7y^3 + 12818x^7y^2 + 7498x^7y + 1652x^7 + 50x^6y^6 + 754x^6y^5 + 3532x^6y^4 + 7452x^6y^3 + 7858x^6y^2 + 4018x^6y + 784x^6 + 77x^5y^6 + 760x^5y^5 + 2737x^5y^4 + 4760x^5y^3 + 4294x^5y^2 + 1918x^5y + 330x^5 + 4x^4y^7 + 101x^4y^6 + 655x^4y^5 + 1837x^4y^4 + 2651x^4y^3 + 2050x^4y^2 + 798x^4y + 120x^4 + 6x^3y^7 + 97x^3y^6 + 464x^3y^5 + 1042x^3y^4 + 1258x^3y^3 + 833x^3y^2 + 280x^3y + 36x^3 + 8x^2y^7 + 77x^2y^6 + 273x^2y^5 + 490x^2y^4 + 490x^2y^3 + 273x^2y^2 + 77x^2y + 8x^2 + 8xy^7 + 49xy^6 + 126xy^5 + 175xy^4 + 140xy^3 + 63xy^2 + 14xy + x + y^8 + 7y^7 + 21y^6 + 35y^5 + 35y^4 + 21y^3 + 7y^2 + y$ }\\
\hline

\end{tabular}
\end{adjustwidth}

\vspace{0.5cm}
\textbf{Appendix 3.} Numbers of spanning trees of pyrene chains and triphenylene chains.
\begin{table}[H]
\centering
\begin{tabular}{|c|c|c|c|c|}
\hline
 &	$n=1$
&$n=2$&$n=3$&$n=4$\\
\hline

$R_n$&	1092
&1150848 &1212779520&1278043619328\\
\hline

$T_n$&	1188
&1369728 &1579253616 &1820830109040\\
\hline
\end{tabular}
\end{table}

\end{document}